\documentclass[12pt,reqno]{amsart}
\usepackage{graphicx}
\usepackage{amsmath}
\usepackage{amssymb}
\usepackage{amscd}
\usepackage{mathrsfs}
\usepackage[all,cmtip]{xy}
 \usepackage{color}

\newtheorem{theorem}{Theorem}

\newtheorem{theoremc}{Theorem}

\newtheorem{rk}[theoremc]{Remark\!}

\newtheorem{lem}[theoremc]{Lemma\!\!}

\newcommand\bib[1]{\bibitem[#1]{#1}}

\renewcommand\a{\alpha}

\newcommand\com[1]{}
\newcommand\C{{\mathbb C}}
\newcommand\Cc{{\let\mathcal\mathscr\mathcal C}}

\renewcommand\d{\delta}

\newcommand\E{\mathcal{E}}

\newcommand\oo{\omega}
\newcommand\op[1]{\mathop{\rm #1}\nolimits}
\newcommand\ot{\otimes}
\newcommand\p{\partial}

\newcommand\R{{\mathbb R}}

\newcommand\we{\wedge}

\begin{document}

 \title[Differential Invariants and Symmetry]{Differential Invariants and Symmetry:\\
 Riemannian metrics and beyond}
 \author{Boris Kruglikov}
 \date{}
 \address{Institute of Mathematics and Statistics, University of Troms\o, Troms\o\ 90-37, Norway.
\quad E-mail: {\tt boris.kruglikov@uit.no}. }
 \keywords{Differential Invariants, Invariant Derivations, Riemannian and pseudo-Riemannian metrics, Killing fields}

 \vspace{-14.5pt}
 \begin{abstract}
We discuss Lie-Tresse theorem for the pseudogroup of diffeomorphisms acting on the space of (pseudo-)Riemannian metrics, and relate this to existence of Killing vector fields. Then we discuss the impact
of symmetry in the general case.
 \end{abstract}

 \maketitle

\section*{Introduction}

Consider a pseudogroup $G$ acting on a differential equation $\E\subset J^\infty(E,n)$, where
the latter denotes the space of jets of dimension $n$ submanifolds $M\subset E$. As a particular important case, we
consider the action on the full/un-constrained space of jets, which includes the action on the
space of geometric structures of a prescribed type on a manifold $M$ of dimension $n$.

Let $s_k$ be the number of independent differential invariants of order $\le k$, which is the codimension
of the general orbit of $G$ in $\E^k\subset J^k$ (the respective jet-space), and let
$\d_k=s_k-s_{k-1}$ be the number of differential invariants of order $k$. The Poincar\'e function
(at a Zariski generic point $q_\infty\in\E$) of the $G$-moduli in $\E$ is the series
 $$
P^\E_G(z)=\sum\d_kz^k.
 $$
It follows from \cite{KL$_2$} that for algebraic actions transitive on the base $J^0=E$ of $\E$ the function
$P(z)$ is rational with the pole at $z=1$ of degree $d$, where $d$ does not exceed
the number of independent invariant derivations involved in the Lie-Tresse representation of the algebra of
(rational-polynomial) differential invariants (and also $d\le$ dimension of the affine characteristic variety
of $\E$, which does not exceed $\dim M$). 

In this paper we will discuss relation of the algebra $\mathfrak{A}$ of differential invariants to symmetries of the
considered equations/geometric structures. The well-studied example, which we discuss in details, is the space
of (pseudo-)Riemannian metrics.

Differential invariants of Riemannian metrics is a classical heritage coming back to Riemann. In his inaugural lecture
\cite{R} he argued for existence of differential invariants when $n\ge2$ and identified the fundamental tensor
invariant -- the curvature $R\in\Gamma(\Lambda^2T^*M\ot\op{End}(TM))$. His argument for existence of invariants
is the functional dimension count: $\binom{n+1}2$ coefficients of the metric minus $n$ functions in the
general change of coordinates yields $\binom{n}2$ scalar invariants of the problem. Riemann alluded also to the form of
these differential invariants, which after proper understanding 
can be identified with the spectrum of the curvature operator $R\in\Gamma(\op{End}(\Lambda^2TM))$.

We will see below that there are more independent scalar differential invariants even of order 2 (for $n>3$;
or of higher order for the general $n>1$). Absence of contradiction with the naive functional dimension count
is due to existence of differential syzygies (Bianchi identities).

Existence of Killing vector fields of a given Riemannian metric $g$ can be seen from the restriction of scalar
differential invariants to the base manifold (realized as the jet-lift via the metric $g$ to $J^\infty$), and we
will demonstrate that this is also the case with the general equations/geometric structures $\E$ acted upon by a
Lie pseudogroup $G$. The notion of regularity of geometric structure (in this context it means separability
via differential invariants) plays the crucial role (as it does in the study of pseudo-Riemannian vs.\ Riemannian
structures). We will also see that existence of symmetries bounds the growth of differential invariants
via the order of pole in the Poincar\'e function.

\section{Background: Lie-Tresse theorem}

Let us recall Lie-Tresse theorem on finite-generation of the (generally infinite) algebra of
differential invariants of a Lie pseudogroup $G$, which in premature formulation belongs to \cite{Tr}.

The global algebraic statement is due to \cite{KL$_2$} and sounds as follows
(see also \cite{OPV} and references therein for other approaches).

Consider a transitive action of a Lie pseudogroup $G$ on a manifold $M$ such that its prolongation is an
algebraic action on a formally integrable irreducible differential equation $\E$ over $M$.

 \begin{theorem}[\cite{KL$_2$}]
There exists a number $l$ and a Zariski closed invariant proper subset $S_l\subset\E^l$ such that
the action is regular in $\pi_{\infty,l}^{-1}(\E^l\setminus S_l)\subset\E$, so that there exists a
rational geometric quotient
 $$
\bigl(\E^k\setminus\pi_{k,l}^{-1}(S_l)\bigr)/G^k\simeq Y_k,\quad \dim Y_k=s_k,\quad \forall k\ge l.
 $$
The algebra $\mathfrak{A}^l$ of differential invariants on $\E$, which are rational by jet-variables
up to order $l$ and polynomial by the jet-variables of higher order,
separates the regular orbits from $\E\setminus\pi_{\infty,l}^{-1}(S_l)$.

There exists a finite number of functions $I_1,\dots,I_t\in\mathfrak{A}^l$ and a finite number
of rational invariant derivations $\nabla_1,\dots,\nabla_s:\mathfrak{A}^l\to\mathfrak{A}^l$
such that any function from $\mathfrak{A}^l$ is a polynomial of differential
invariants $\nabla_{j_1}\cdots\nabla_{j_r}I_i$ with coefficients being rational functions of $I_i$.
 \end{theorem}

In the next section we will show specification of this theorem for Riemannian metrics.
Classification of all scalar differential invariants implies recognition problem
for (Zariski generic) Riemannian metrics.

The equivalence problem (covering all metrics) is quite different:
Following E.\,Cartan's method of equivalence \cite{C}, one associates the oriented orthonormal frame bundle
$\mathcal{F}_M$ over $M$. This is a principal $SO(n)$-bundle with the Cartan connection
$(\theta,\oo)\in\Omega^1(\mathcal{F}_M,\mathfrak{euc}(n))$, where
$\mathfrak{euc}(n))=\R^n\rtimes\mathfrak{so}(n)$ is the Lie algebra of the Euclidean motion group.
The structure equations
 $$
d\theta+\omega\we\theta=0,\quad d\omega+\omega\we\omega=R\,\theta\we\theta
 $$
carry the fundamental invariant -- the curvature tensor $R$, as well as the canonical parallelism
resolving the equivalence for metrics.

They also encode the complete set of scalar differential invariants, the so-called Weyl invariants.
These are obtained from the tensor products of the (Levi-Civita) covariant derivatives of $R$ by total contractions.
The fundamental theorem \cite{W} states that all scalar differential invariants that are polynomial in
jets of order $>0$ are linear combinations of Weyl invariants. Application of this to recognition problems
for metrics is given in \cite[Note 19]{KN}.

Note however that this approach does not give the number of independent Weyl invariants, nor their form
(i.e. prescription which contractions shall be considered). Explicit description of scalar differential
invariants will be discussed in the next section.

\section{Riemannian metrics: Computing the invariants}
\label{S2}

Let $E=S^2_\text{ndg}T^*M$ be the fiber bundle of non-degenerate quadratic forms over $M$.
Its space of sections consists of pseudo-Riemannian metrics on $M$.
If we change to positive definite forms, we get Riemannian metrics, and we will not make distinction
between them in this section.

The pseudogroup $G=\op{Diff}_\text{loc}(M)$ of diffeomorphisms is algebraic and
acts transitively on the base $M$. Its differential invariants in $J^\infty E$ are invariants of metrics.
Thus our Lie-Tresse theorem (and also \cite{MM$_2$}) implies existence of the generating set of invariants separating
generic metrics (i.e.\ with a Zariski generic jet; this does not concern all Riemannian metrics, for instance,
there are VSI metrics in the general relativity that are non-closed orbits of the action, see \cite{H}).

The number of independent differential invariants $s_k$ of order $\le k$ was computed by T.\,Thomas
\cite[\S75]{Th}, see also \cite{MM$_1$}: there are indeed no invariants on the first jets $s_0=s_1=0$ and
 $$
s_k=\left\{
\begin{array}{ll}
\frac{(k+1)(k-2)}2+\delta_{k,2}, & \text{ for }n=2,\vphantom{\frac{\frac22}{2^2}}\\
n+\frac{n^2(k-1)-n(k+1)}{2(k+1)}\cdot\binom{n+k}k, & \text{ for }n>2.\vphantom{\frac{\frac22}{2^2}}
\end{array}
\right.
 $$
The proof is based on the fact that the action of the differential group $G_x^{k+1}$ of order $k+1$
(consisting of jets of diffeomorphisms preserving $x$) on the space of jets $J^k_xE$ over $x\in M$
is free for $k\ge3$ when $n=2$ and for $k\ge2$ when $n\ge3$, thus implying the claim. This in turn yields the
number of scalar invariants of pure order $k$:
$\d_0=\d_1=0$,
 $$
\d_k=\left\{
\begin{array}{ll}
k-1-\delta_{k,3}, & \text{ for }n=2,\vphantom{\frac{\frac22}{2^2}}\\
n+\frac{(n+2)(n+1)n(n-3)}{12}, & \text{ for }k=2,n>2,\vphantom{\frac{\frac22}{2^2}}\\
\frac{n(k-1)}2\cdot\binom{n+k-1}{k+1}, & \text{ for }k>2,n>2.\vphantom{\frac{\frac22}{2^2}}\end{array}
\right.
 $$
Thus the Poincar\'e function is equal to
 $$
P(z)=\left\{
\begin{array}{ll}
\frac{z^2(1-z+2z^2-z^3)}{(1-z)^2}, & \text{ for }n=2,\vphantom{\frac{\frac22}{\frac22}}\\
\frac{n}z+\binom{n}2\cdot(1-z^2)-\frac1{(1-z)^n}\cdot\bigl(\frac{n}z-\binom{n+1}2\bigr), & \text{ for }n>2
\vphantom{\frac{\frac22}{\frac22}}
\end{array}
\right.
 $$
(NB: there are no pole at $z=0$ in the last rational function).
Notice also that the generating function of $s_k$ is $Q(z)=\sum s_kz^k=P(z)/(1-z)$.

The explicit form of differential invariants for $n=2$ was written down in \cite{K$_1$}. Let us discuss the case $n\ge3$. We begin with the order $k=2$. In this case the fundamental tensorial invariant of metric $g$ is the curvature tensor
$R=R_g$, having orthogonal decomposition $R=\op{Ric}+W$ into Ricci (including trace) and Weyl parts.

The Ricci tensor $\op{Ric}\in\Gamma(S^2T^*M)$ gives $n$ invariants as follows. Construct the Ricci operator
 $$
A=g^{-1}\op{Ric}:TM\to TM
 $$
and take its traces $I_i=\op{Tr}(A^i)$, $i=1,\dots,n$. For generic $g$ these are $n$ independent
functions forming coordinates on $M$.

 \begin{rk}
It was noticed by V.Lychagin and A.Kotov that the invariants $I_1,\dots,I_n$ suffice for classification of generic
Riemannian metrics due to the principle of $n$-invariants as discussed in \cite{ALV}.
 \end{rk}

Alternatively, one can take eigenvalues of $A$ for $I_i$. Though the first approach is preferable due to rational output
(we have to allow algebraic extension when computing eigenvalues), let us proceed for an instant with the second
approach. Denote the eigenvectors of $A$ by $e_i$ (for generic $g$ the operator $A$ is semi-simple) and take the dual co-frame $\theta^i$ (these are vector fields in total derivatives and horizontal differentials respectively).
Then we can decompose $W=W^i_{jkl}e_i\ot\theta^j\ot\theta^k\we\theta^l$, and take the components $W^i_{jkl}$ as the other scalar invariants of $R$. These together with $I_i$ form the complete set of invariants of order 2.

To skip using the algebraic extensions in finding scalar invariants of order $k=2$ we proceed as follows.
For $n=3$ the Riemannian curvature is encoded by the Ricci tensor, and $I_1,I_2,I_3$ are
the only invariants. For $n\ge4$ there is also Weyl tensor, viewed as a map $W:\Lambda^2TM\to\Lambda^2TM$.
Introducing the operators
 $$
W^{a,b,c}=A^a\circ W^b\circ A^c:\Lambda^2TM\to\Lambda^2TM
 $$
we can take the traces $\op{Tr}(W^{a,b,c})$ as
differential invariants $J_1,\dots,J_r$, $r=\frac1{12}(n+2)(n+1)n(n-3)$ of order 2
(we can restrict to $1\le a,c\le n$ and $1\le b\le\binom{n}2$, these traces contain $r$ independent
invariants, and there will be some relations with the others). Now $\{I_1,\dots,I_n,J_1,\dots,,J_r\}$ form a
complete set of invariants of order $k=2$.

To obtain differential invariants of order $k>2$ let us consider the horizontal differentials
$\hat d I_1,\dots,\hat d I_n$. The dual basis of vector fields in total derivatives
$\nabla_i=\hat\p/\hat\p I_i$, $i=1,\dots,n$, are called Tresse derivatives \cite{KL$_1$}.
The basis of higher order scalar invariants is contained in the set
$(\nabla_{i_1}\dots\nabla_{i_{k-2}}R)(A^{s_1}\nabla_{j_1},A^{s_2}\nabla_{j_2},A^{s_3}\nabla_{j_3},A^{s_4}\nabla_{j_4})$,
$1\le i_p,j_q,s_t\le n$.
The Bianchi identities (and their covariant derivatives) yield the relations among these invariants.

\section{Differential invariants and symmetries: metric case}
\label{S3}

In the case of (pseudo-)Riemannian structures the $G$-orbits through $q_\infty\in\E$ with
non-trivial stabilizers are not regular\footnote{There are several different notions of regularity,
see Section \ref{S3+}.} (this means the property of having a Killing field is not generic).
The same holds for general $G$-actions, with the freeness property as in the previous section.

Consider restrictions of differential invariants to the structure, viewed as
the jet-lift of $M$ to the space of infinite jets, so that they become functions on the base $M$.
We begin with the metric case, and denote this restriction of the algebra $\mathfrak{A}\subset C^\infty(J^\infty E)$
to $(M,g)$ by $\mathfrak{A}_g\subset C^\infty(M)$.

It is clear that the algebra $\mathfrak{A}_g$ is annihilated by any symmetry (Killing vector field) of the
metric structure $g$. For Riemannian metrics one can obtain the converse.
The following result is a refinement of Singer's theorem \cite{Si} (which concerns tensorial invariants $\nabla^sR$)
characterizing existence of a (local) transitive isometry group of $(M,g)$.

 \begin{theorem}[\cite{PTV}]\label{Thm2}
A Riemannian manifold $(M^n,g)$ is locally homogeneous iff all scalar Weyl invariants of order
$s\le\binom{n}2$ are constant.
 \end{theorem}

Here is a generalization due to Console and Olmos about existence of some Killing fields.
Recall that homogeneity is the dimension of the regular orbit in $M$ of the full isometry group of $g$.

 \begin{theorem}[\cite{CO}]\label{Thm3}
The homogeneity of a Riemannian manifold $(M,g)$ is equal to dimension of a generic
level set of the Weyl invariants.
 \end{theorem}

This latter theorem is true for all Riemannian metrics and regular pseudo-Riemannian metrics,
but it can fail for non-generic pseudo-Riemannian metrics.
Indeed, in \cite{KM} an example of a Lorenzian signature Einstein 4D metric is given with all Weyl invariants zero,
yet without any Killing vector field. This structure belongs to the singular set of metrics that are
not separated by curvature invariants.

\section{Differential invariants and symmetries: general case}
\label{S3+}

Let us consider now the general differential equation/space of geometric structures $\E$ acted upon by
a Lie pseudogroup $G$. Suppose a Lie (pseudo-)group $H\subset G$ acting on $M$ with regular orbits of $\dim=k$
is a symmetry of all solutions/structures from $\E'\subset\E$, i.e. there is a common stabilizer $H$ for any
$q\in\op{Sol}(\E')$; the latter mean the set of all holonomic sections $q=j_\infty(s)$ with the image in $\E'$
(we will assume for simplicity that $\E'$ is compatible/involutive and locally solvable near a generic point,
which always hold for analytic or finite type systems). Thus $q$ is either a section of a vector bundle $E$
subject to a differential constraint $\E$ (\cite{KL$_1$}) or a geometric structure
(usual $G$-structure or more general filtered geometric structure in the sense of \cite{Mo,K$_2$}).

 \begin{lem}
Invariant derivations in the direction of $\mathfrak{h}=\op{Lie}(H)$ act trivially on the space $\mathfrak{A}|_{\E'}$
of differential invariants restricted to $\E'$.
 \end{lem}

 \begin{proof}
Let $\nabla$ be an invariant derivation and $I$ a scalar differential invariant. By \cite{KL$_2$} $\nabla$ can be
represented by an invariant horizontal vector field $X$ (field in total derivatives) modulo non-invariant part that
acts trivially. Then the value of $\nabla I$ at $q_\infty\in\E'$ can be interpreted as follows \cite{KL$_1$}.
Choose a holonomic section $q$ representing $q_\infty$ at the point $q_0$ and restrict both $X$ and $I$ on $M$
embedded as the corresponding section in $\E'$; denote the resulting field and invariant by $\bar X$, $\bar I$.
We get: $(\nabla I)(q_\infty)=(X\cdot I)(q_\infty)=L_{\bar X}({\bar I})(q_0)$. The claim follows.
 \end{proof}

From this lemma we conclude that the order of pole of the Poincar\'e function $P_G^{\E'}(z)$ (counting the differential invariants $\mathfrak{A}|_{\E'}$) at $z=1$ satisfies $d\le n-k$.
We can interpret $d$ as the functional rank of the restriction $\mathfrak{A}_q$ of the algebra of differential
invariants to a generic solution/structure $q\in\op{Sol}(\E')$, then the inequality becomes evident.

By the functional rank above we mean the maximum $m$ such that the general stratum of the space $\mathfrak{A}_q$
is parametrized by some number of arbitrary functions of $m$ variables. The differential invariants are
governed by a certain differential equation (factor-equation of \cite{KL$_2$}), whose solution space is
controlled by the affine characteristic variety $\op{Char}^\C_\text{aff}$ of dimension $d$.
Symmetries (point symmetries in the case of a differential equation $\E$) give linear hyperplanes
containing $\op{Char}^\C_\text{aff}$, whence the restriction on $d$.

To the opposite side we have the following assertion. Call a solution/structure $q\in\op{Sol}(\E)$ regular if
its $\infty$-jet belongs to the Zariski open domain, where points are separated by differential invariants of
the pseudogroup $G$ (notice that this is a weaker notion than regularity of the orbit in the orbit space;
in particular we do not require that the orbit through $q$ in a respective jet-space has maximal dimension,
though of course to have separation property the orbit must be closed).

We call a solution/structure $q$ rigid if the Lie equation $\mathfrak{Lie}(q)$,
describing the symmetry of $q$ \cite{KS}, has finite type (for instance, pseudo-Riemannian metrics,
projective structures in any dimension and conformal structures in dimension $n>2$ are such).
In this case the maximal (local and global) symmetry group of $q$ is finite-dimensional.

 \begin{theorem}\label{Thm4}
Suppose that the solutions/structures of $\E$ are either rigid or that they are analytic\footnote{The former
is encoded in the formal geometry of $\E$, while the latter can be guaranteed by the Holmgren theorem \cite{Se}.}.
Then if the restriction $\mathfrak{A}_q$ of the algebra of scalar differential invariants $\mathfrak{A}$ to
a regular solution/structure $q$ has functional rank $m$, there exists a symmetry algebra of $q$
(possibly after shrinking the domain) with regular orbits on $M$ of $\dim\op{Sym}(q)=n-m$.
 \end{theorem}

Notice that regularity assumption is essential, otherwise one gets counter-examples as discussed after
Theorem \ref{Thm3}.

 \begin{proof}
Consider the Lie equation $\mathfrak{Lie}(q)$ on the symmetry of $q\in\op{Sol}(\E)$. Its Cartan-Kuranishi
completion $\overline{\mathfrak{Lie}}(q)$ \cite{KLV,Se} is regular over an open dense set in $M$
(cf. \cite[Proposition 2]{K$_2$}). Thus in both cases -- finite type or analyticity -- the compatibility in an open
domain of regular points shall be checked on a finite number of steps, the maximum of which we denote by $k$.
Existence of a finite jet equivalence is given by the equality of scalar differential invariants to that jet-order
$k$, which holds by our assumption: $k$-jets of the transformations preserving the algebra of invariants
$\mathfrak{A}^k_q$ of order $\le k$ form the jet sub-pseudogroup $H^k\subset G^k$. Integrating the involutive equation
$\overline{\mathfrak{Lie}}(q)$ (which is possible by our assumptions) we obtain the space $H$ of transformations
preserving the solution/structure $q$. In other words, $H$ is the stabilizer $\op{Sym}(q)=\{g\in G:g\cdot q=q\}$.
By construction its orbits are level lines of the differential invariants $\mathfrak{A}_q$. Whence the claim.
 \end{proof}

For Cartan geometries \cite{Sh} we sketch another approach. Such geometric structures\footnote{Riemannian structure
is a particular case of the Cartan geometry.} $q$ are encoded as
a principal $P$-bundle $\mathcal{F}\to M$ with a canonical coframe -- Cartan connection
$\omega\in\Omega^1(\mathcal{F},\mathfrak{g})$, which is equivariant with respect to the right $P$-action.
The space of invariants $\hat{\mathfrak{A}}_q$ of the parallelism on $\mathcal{F}$ is formed by the structure functions
and their iterated derivatives with respect to the (dual) frame.
The maximal number of functional independent among them is the rank $r$ of the coframe $\{\oo^\a\}$.
Let $P'\subset P$ be the subgroup of the structure group that acts by symmetry on $\hat{\mathfrak{A}}_q$.
Absolute invariants of $q$ are $P$-invariant functions from $\hat{\mathfrak{A}}_q$, which thus descend to functions on
$M$. Thus $m=r-(\dim P-\dim P')$, and the conclusion follows from \cite[Theorem 14.26]{O}:
dimension of the symmetry orbit in $\mathcal{F}$ equals the corank of the structure functions of the coframe
$n+\dim P-r=n+\dim P'-m$ and since the vertical symmetry (stabilizer) is $P'$, the regular orbits of the symmetry group
on $M$ have dimensions $n-m$.

It shall be noted that this proof generalizes to the filtered analytic structures of infinite type
using the tower construction of Morimoto \cite{Mo} and the stabilization by involutivity.


\smallskip

\section{Classification in the presence of symmetry. Conclusion}
\label{S4}

In the presence of symmetry $H$ the classification of solutions/struc\-tures $q$ from $\E$ is modified as follows.
The differential invariants restricted to $M$ via $q$ descend to the (local or global depending on the setup) quotient
$N=M/H$. Then these scalar invariants form "coordinate set" and the equivalence proceeds as before,
but in addition one has to encode the $H$-principal bundle over $N$ and the homogeneous structure on $H$,
which both are finite type data.

For the model example of Riemannian metrics with the isometry group $H$ the bundle has the natural Ehresmann connection
-- the orthogonal complement to the fiber (for pseudo-Riemannian metrics one has to require non-degeneracy of
the metric restriction to the orbits), which is $H$-equivariant and so is a principal $H$-bundle connection.
In addition one has to add the metric on the Lie algebra $\mathfrak{h}=\op{Lie}(H)$, and these together with the basis
of scalar differential invariants of the base $(N,\bar g)$ defined as in Section \ref{S2} complete the classification.

\medskip

To finish the paper let us mention some open problems, which are left beyond the scope of this paper.

\smallskip

{\bf Question 1:} What are the differential syzygies of the generators of scalar differential invariants for Riemannian metrics as indicated in Section \ref{S2} (for $n=2$ computed in \cite{K$_1$}) and their Poincar\'e function?

\smallskip

{\bf Question 2:} What is the Poincar\'e function of the stratum of Riemannian metrics admitting the isometry algebra of $\dim=k$ orbit? This function should have a pole at $z=1$ of order $n-k$.

\smallskip

{\bf Question 3:} In which cases Theorem \ref{Thm4} characterizing existence of symmetries in the regular region hold in any region (like Riemannian vs. Lorenzian case for the metric structures)?


\end{document}